\documentclass[oneside,final,11pt]{amsart}

\usepackage{dsfont}

\newcommand{\R}{\mathds R}
\newcommand{\dd}{\mathrm d}
\newcommand{\EP}{$c_0$EP}
\newcommand{\DA}{\mathrm{DA}}

\newcommand{\upstar}[1]{{#1}^{\raise1pt\hbox{$\scriptscriptstyle*$}}}
\newcommand{\chilow}[1]{\chi_{\lower2pt\hbox{$\scriptstyle#1$}}}

\newcommand{\notarrow}{\mathrel{\hbox{$\hspace{3pt}\not{}\hspace{-10pt}\longrightarrow$}}}

\DeclareMathOperator{\Ker}{Ker}
\DeclareMathOperator{\BV}{BV}
\DeclareMathOperator{\NBV}{NBV}

\DeclareMathOperator{\supp}{supp}

\title[On the $c_0$-extension property for compact lines]{On the $\mathbf{c_0}$-extension property for\\
compact lines}

\author{Claudia Correa}
\thanks{The first author is sponsored by FAPESP (Process no.\ 2012/25171-0).}
\address{Departamento de Matem\'atica,\hfill\break\indent Universidade de S\~ao Paulo, Brazil}
\email{claudiac.mat@gmail.com}
\author{Daniel V. Tausk}
\address{Departamento de Matem\'atica,\hfill\break\indent Universidade de S\~ao Paulo, Brazil}
\email{tausk@ime.usp.br} \urladdr{http://www.ime.usp.br/\~{}tausk}

\date{March 3rd, 2014}

\begin{document}

\theoremstyle{plain}\newtheorem{teo}{Theorem}[section]
\theoremstyle{plain}\newtheorem{prop}[teo]{Proposition}
\theoremstyle{plain}\newtheorem{lem}[teo]{Lemma}
\theoremstyle{plain}\newtheorem{cor}[teo]{Corollary}
\theoremstyle{definition}\newtheorem{defin}[teo]{Definition}
\theoremstyle{remark}\newtheorem{rem}[teo]{Remark}
\theoremstyle{plain} \newtheorem{assum}[teo]{Assumption}
\theoremstyle{definition}\newtheorem{example}[teo]{Example}

\begin{abstract}
We present a characterization of the continuous increasing surjections $\phi:K\to L$ between compact lines $K$ and $L$
for which the corresponding subalgebra $\phi^*C(L)$ has the $c_0$-extension property in $C(K)$.
A natural question arising in connection with this characterization is shown to be independent
of the axioms of ZFC.
\end{abstract}

\maketitle

\begin{section}{Introduction}

In this paper we continue the study of the $c_0$-extension property in the context of spaces
of the form $C(K)$ which was initiated in \cite{CTJFA, CT}. Here, as usual, $C(K)$ denotes the space of continuous real-valued functions on a compact Hausdorff space $K$, endowed with the supremum norm. A closed subspace $Y$ of a Banach space $X$ is said to have
the {\em $c_0$-extension property\/} (briefly: \EP) in $X$ if every bounded $c_0$-valued operator defined in $Y$ admits a bounded extension to $X$. This definition was introduced by the authors in \cite{CT} with the purpose of
studying extensions of the celebrated Theorem of Sobczyk \cite{Sobczyk}, which states that every closed subspace
of a separable Banach space $X$ has the \EP\ in $X$. The quest for generalizations of Sobczyk's Theorem has been
engaged upon by many authors \cite{ArgyrosLondon, CTJFA, CT, Molto, Patterson, c0sum, JesusSobczyk}. For instance, we
have used in \cite{CTJFA} the notion of \EP\ to prove that every isomorphic copy of $c_0$ in $C(K)$ is complemented,
when $K$ is a compact line. By a {\em compact line\/} we mean a linearly ordered set which is compact in the order topology. Topological properties of compact lines and structural
properties of their spaces of continuous functions have recently been studied in a series of articles \cite{CTJFA, CT, KK, KK2, Kubis}.

The main result of this paper (Theorem~\ref{thm:main}) is a characterization of the continuous increasing
surjections $\phi:K\to L$ between compact lines $K$ and $L$ such that the range $\phi^*C(L)$ of the composition
operator
\[\phi^*:C(L)\ni f\longmapsto f\circ\phi\in C(K)\]
has the \EP\ in $C(K)$. When the latter condition holds, we say that the map $\phi$ has the \EP. The characterization given in Theorem~\ref{thm:main} involves the order structure of $L$ and the set $Q(\phi)$ defined as:
\[Q(\phi)=\big\{t\in L:\vert\phi^{-1}(t)\vert>1\big\},\]
where $\vert\cdot\vert$ denotes the cardinality of a set. We note that in \cite[Lemma~2.7]{KK} it is given a necessary and sufficient condition for the complementation of $\phi^*C(L)$ in $C(K)$ (equivalently,
for the existence of an averaging operator for $\phi$), again in terms of the order structure of $L$ and the set $Q(\phi)$.

This paper is organized as follows. In Section~\ref{sec:c0EP}, we prove a criterion (Proposition~\ref{thm:extcriteria}) for the extensibility to $C(K)$ of $c_0$-valued bounded operators defined in $\phi^*C(L)\equiv C(L)$. This criterion is
a generalization of \cite[Lemma~3.2]{CTJFA} and it is used in the proof of Theorem~\ref{thm:main}. In Section~\ref{sec:separable}, we present a nicer formulation (Theorem~\ref{thm:mainseparavel}) of the main result of the paper in the case when $L$ is separable. A naturally arising question is then considered and shown to be independent of the axioms of ZFC.

\end{section}

\begin{section}{Characterization of increasing maps with the \EP}
\label{sec:c0EP}

We start by fixing the terminology and notation for the paper and by recalling some elementary facts.
Given a compact Hausdorff space $K$, we identify as usual the dual space of $C(K)$ with the space $M(K)$ of finite countably-additive signed regular Borel measures on $K$, endowed with the total variation norm $\Vert\mu\Vert=\vert\mu\vert(K)$. Given a point $p\in K$,
we denote by $\delta_p\in M(K)$ the probability measure with support $\{p\}$. If $\phi:K\to L$ is a continuous
map between compact Hausdorff spaces $K$ and $L$, then the adjoint of the composition operator $\phi^*$
is denoted by $\phi_*:M(K)\to M(L)$ and it is given by:
\[\phi_*(\mu)(B)=\mu\big(\phi^{-1}[B]\big),\]
for every $\mu\in M(K)$ and every Borel subset $B$ of $L$.

Bounded operators $T:C(K)\to\ell_\infty$ are always identified with bounded sequences of measures $(\mu_n)_{n\ge1}$ in $M(K)$, where $\mu_n$
represents the $n$-th coordinate functional of $T$. In this case, we will say that $T$ is associated with $(\mu_n)_{n\ge1}$.
Note that $T$ takes values in $c_0$ if and only if $(\mu_n)_{n\ge1}$ is weak*-null.

Let $X$ be a linearly ordered set and $A$ be a subset of $X$. A point $x\in X$ is said to be a {\em right limit point\/} (resp., {\em right condensation point}) of $A$ (relatively to $X$) if $x$ is not the maximum
of $X$ and for every $y\in X$ with $y>x$ we have that $\left]x,y\right[\cap A$ is nonempty (resp., is uncountable).
A point of $A$ that is not a right limit point of $A$ is said to be {\em right isolated\/} in $A$. Similarly, one defines left limit, left condensation and left isolated points. If $x\in X$ is a two-sided limit point
of $X$ (i.e., $x$ is both a left limit point and a right limit point of $X$), then we call $x$ an {\em internal\/} point of $X$. The points of $X$ that are not internal are called {\em external}.
Given a compact line $K$, note that if $H$ is a closed subset of $K$ and $A$ is a subset of $H$, then a point $t\in H$ is a right limit point (resp., left limit point) of $A$ relatively to $H$ if and only if $t$ is a right limit point (resp., left limit point) of $A$ relatively to $K$.

When $K$ is a compact line, it is possible to give a more concrete description of the space $M(K)$
in terms of functions of bounded variation (Lemma~\ref{thm:MKNBVK} below). Given a map $F:K\to\R$, the total variation $V(F)\in[0,+\infty]$ is defined exactly as in the case $K=[0,1]$. We denote by $\BV(K)$ the Banach space of functions $F:K\to\R$ of bounded variation
(i.e., functions $F$ with $V(F)<+\infty$) endowed with the norm:
\[\Vert F\Vert_{\BV}=\vert F(0)\vert+V(F),\]
where $0$ always denote the minimum element of a compact line.
Then:
\[\NBV(K)=\big\{F\in\BV(K):\text{$F$ is right-continuous}\big\}\]
is a closed subspace of $\BV(K)$.
\begin{lem}\label{thm:MKNBVK}
Given a compact line $K$, the map:
\begin{equation}\label{eq:muFmu}
M(K)\ni\mu\longmapsto F_\mu\in\NBV(K)
\end{equation}
is a linear isometry, where $F_\mu$ is defined by $F_\mu(t)=\mu\big([0,t]\big)$, for all $t\in K$.
\end{lem}
\begin{proof}
See \cite[Lemma~3.1]{CTJFA}.
\end{proof}

Throughout the remainder of the paper, we always denote by $\phi:K\to L$ a continuous increasing surjection between compact lines $K$ and $L$. We set:
\[Q_0(\phi)=\big\{t\in Q(\phi):\text{$t$ is an internal point of $L$}\big\}.\]

Our first goal is to prove the extension criterion for $c_0$-valued bounded operators defined in $\phi^*C(L)\equiv C(L)$ (Proposition~\ref{thm:extcriteria}). The proof requires two lemmas.

\medskip

By the {\em canonical retraction\/} of $K$ onto a given closed interval $[a,b]$ of $K$,
we mean the retraction $R:K\to[a,b]$ that maps $[0,a]$ to $a$ and $[b,\max K]$ to $b$. Then $R^*C\big([a,b]\big)$
consists of the elements of $C(K)$ that are constant on $[0,a]$ and on $[b,\max K]$.
\begin{lem}\label{thm:lindense}
For each $t\in L$,
denote by $R_t:K\to\phi^{-1}(t)$ the canonical retraction of $K$ onto the closed interval $\phi^{-1}(t)$. Then
the set:
\begin{equation}\label{eq:lindenso}
\phi^*C(L)\cup\!\!\!\bigcup_{t\in Q(\phi)}\!\!R_t^*C\big(\phi^{-1}(t)\big)
\end{equation}
is linearly dense in $C(K)$, i.e., it spans a dense subspace of $C(K)$.
\end{lem}
\begin{proof}
By \cite[Proposition~3.2]{Kubis}, the set of continuous increasing functions is linearly dense in $C(K)$
and therefore it suffices to pick an arbitrary increasing function $f\in C(K)$ and prove that it belongs to
the closed linear span of \eqref{eq:lindenso}. For $t\in Q(\phi)$, write $\phi^{-1}(t)=[a_t,b_t]$ and
let $g_t\in R_t^*C\big(\phi^{-1}(t)\big)$ be the map that agrees with $f-f(a_t)$ on $[a_t,b_t]$. Then:
\[\sum_{t\in Q(\phi)}\Vert g_t\Vert=\sum_{t\in Q(\phi)}\big(f(b_t)-f(a_t)\big)\le f(\max K)-f(0)<+\infty,\]
and therefore the sum $\sum_{t\in Q(\phi)}g_t$ converges to some $g\in C(K)$, which belongs to the
closed linear span of $\bigcup_{t\in Q(\phi)}R_t^*C\big(\phi^{-1}(t)\big)$. Moreover, $f-g$ is constant
on $[a_t,b_t]$, for all $t\in Q(\phi)$, and hence $f-g\in\phi^*C(L)$.
\end{proof}

\begin{cor}\label{thm:corweakstar}
Let $\mu\in M(K)$ and $(\mu_i)_{i\in I}$ be a bounded net in $M(K)$. Then $\mu_i\xrightarrow{\;\text{w*}\;}\mu$
if and only if
$\phi_*(\mu_i)\xrightarrow{\;\text{w*}\;}\phi_*(\mu)$ and $(R_t)_*(\mu_i)\xrightarrow{\;\text{w*}\;}(R_t)_*(\mu)$,
for all $t\in Q(\phi)$, where $R_t$ is defined as in the statement of Lemma~\ref{thm:lindense}.
\end{cor}
\begin{proof}
It follows directly from Lemma~\ref{thm:lindense}, observing that a bounded net converges in the weak*-topology
if and only if it converges at the points of a linearly dense set.
\end{proof}

\begin{lem}\label{thm:externalpoints}
Given a weak*-null sequence $(F_n)_{n\ge1}$ in $\NBV(L)$, the set of external points
$t$ of $L$ such that $F_n(t)\notarrow0$ is countable.
\end{lem}
\begin{proof}
If $t$ is right isolated in $L$, then $[0,t]$ is clopen in $L$ and therefore $F_n(t)\longrightarrow0$. If $t\ne0$ is left isolated in $L$, then it admits an immediate predecessor $t^-\in L$, which is right isolated.
The fact that $F_n$ has bounded variation implies that $F_n(t)=F_n(t^-)$ for all but a countable number
of left isolated points $t\ne0$ of $L$. The conclusion follows.
\end{proof}

\begin{prop}\label{thm:extcriteria}
Let $T:C(L)\to c_0$ be a bounded operator associated with a weak*-null sequence $(F_n)_{n\ge1}$ in $\NBV(L)$. The following conditions are equivalent:
\begin{itemize}
\item[(a)] there exists a bounded operator $T':C(K)\to c_0$ with $T'\circ\phi^*=T$;
\item[(b)] the set of points $t\in Q(\phi)$ such that $F_n(t)\notarrow0$ is countable;
\item[(c)] the set of points $t\in Q_0(\phi)$ such that $F_n(t)\notarrow0$ is countable;
\item[(d)] there exists a bounded operator $T':C(K)\to c_0$ with $T'\circ\phi^*=T$ and
$\Vert T'\Vert\le2\Vert T\Vert$.
\end{itemize}
\end{prop}
\begin{proof}
The equivalence between (b) and (c) follows directly from Lemma~\ref{thm:externalpoints}. For all $t\in L$,
write $\phi^{-1}(t)=[a_t,b_t]$. Now assume (a) and
let us prove (b). The operator $T'$ is associated with a weak*-null sequence $(F'_n)_{n\ge1}$ in $\NBV(K)$
and the equality $T'\circ\phi^*=T$ is equivalent to $F'_n(b_t)=F_n(t)$, for all $n\ge1$ and all $t\in L$.
If $\mu'_n\in M(K)$ corresponds to $F'_n$ through \eqref{eq:muFmu}, then $\sum_{t\in L}\vert\mu'_n\vert\big([a_t,b_t]\big)\le\Vert\mu'_n\Vert<\infty$ and therefore the set:
\begin{equation}\label{eq:countable}
\bigcup_{n=1}^\infty\big\{t\in L:\vert\mu'_n\vert\big([a_t,b_t]\big)>0\big\}
\end{equation}
is countable. Given $t\in Q(\phi)$ not in \eqref{eq:countable}, we claim that $F_n(t)\longrightarrow0$.
To prove the claim, let $f\in C(K)$ satisfy $f\vert_{[0,a_t]}\equiv1$ and $f\vert_{[b_t,\max K]}\equiv0$, and note that:
\[F_n(t)=F'_n(b_t)=\int_Kf\,\dd\mu'_n\longrightarrow0.\]
Now assume (b) and let us prove (d). The set:
\[E=\big\{t\in Q(\phi):F_n(t)\notarrow0\big\},\]
is countable.
For each $n\ge1$, set $G_n=F_n\circ\phi$. Then $G_n$ is in $\NBV(K)$ and
$\Vert G_n\Vert_{\BV}=\Vert F_n\Vert_{\BV}$. Let $S:C(K)\to\ell_\infty$ be the bounded operator associated
with the bounded sequence $(G_n)_{n\ge1}$, so that $\Vert S\Vert=\Vert T\Vert$. Note that $G_n(b_t)=F_n(t)$, for all $n\ge1$ and
$t\in L$, from which it follows that $S\circ\phi^*=T$. We will show that the quotient $C(K)/S^{-1}[c_0]$ is separable and it will follow from \cite[Proposition~2.2,~(a)]{CT} that the operator $S\vert_{S^{-1}[c_0]}$ admits an extension $T':C(K)\to c_0$ with $\Vert T'\Vert\le2\Vert S\Vert$. To conclude the proof of the proposition, it suffices to check that
the image of \eqref{eq:lindenso} under the quotient map $C(K)\to C(K)/S^{-1}[c_0]$ is separable.
Note first that $\phi^*C(L)$ is contained in $S^{-1}[c_0]$. Our plan is to show that the image of
$R_t^*C\big(\phi^{-1}(t)\big)$ under the quotient map $C(K)\to C(K)/S^{-1}[c_0]$ is finite-dimensional,
for all $t\in L$, and that $R_t^*C\big(\phi^{-1}(t)\big)\subset S^{-1}[c_0]$, for all $t\in Q(\phi)\setminus E$.

If $\nu_n\in M(K)$ corresponds to $G_n$ through \eqref{eq:muFmu},
then a simple computation yields:
\[(R_t)_*(\nu_n)=F_n(t)\delta_{a_t}+\big(F_n(\max L)-F_n(t)\big)\delta_{b_t}\in M\big(\phi^{-1}(t)\big),\]
for all $t\in L$. Thus, for $g\in C\big(\phi^{-1}(t)\big)$, we have:
\begin{equation}\label{eq:SRtg}
S\big(R_t^*(g)\big)=\Big(F_n(t)g(a_t)+\big(F_n(\max L)-F_n(t)\big)g(b_t)\Big)_{\!n\ge1}\in\ell_\infty.
\end{equation}
From \eqref{eq:SRtg} it follows that $R_t^*C\big(\phi^{-1}(t)\big)\subset S^{-1}[c_0]$, for all $t\in Q(\phi)\setminus E$,
and that $R_t^*(g)\in\Ker(S)\subset S^{-1}[c_0]$, for all $t\in L$ and all $g\in C\big(\phi^{-1}(t)\big)$ satisfying $g(a_t)=g(b_t)=0$.
This concludes the proof.
\end{proof}

We now state the main result of the paper. Its proof requires several technical lemmas and is left to the end
of this section.
\begin{teo}\label{thm:main}
Let $K$ and $L$ be compact lines. A continuous increasing surjection $\phi:K\to L$ has the \EP\ if and only if the following condition holds: for every separable $G_\delta$ subset
$A$ of $L$, if every point of $A$ is a two-sided limit point of $A$ (relatively to $L$),
then $A\cap Q(\phi)=A\cap Q_0(\phi)$ is countable.
\end{teo}

The proof of the following lemma is a simple adaptation of the proof of \cite[Lemma~2.5]{KK}.
\begin{lem}\label{thm:lemapsi}
If $H$ is an uncountable separable compact line, then there exists a continuous increasing surjection $\psi:H\to[0,1]$
such that $\psi^{-1}(u)$ is countable, for all $u\in[0,1]$.
\end{lem}
\begin{proof}
Define an equivalence relation $\sim$ on $H$ by stating that $t_1\sim t_2$ if and only if the closed interval
with endpoints $t_1$ and $t_2$ is countable. The separability of $H$ implies that the equivalence classes
of $\sim$ are countable closed intervals. Thus, there exists a unique linear order on the quotient $H/{\sim}$
such that the quotient map $\psi:H\to H/{\sim}$ is increasing and continuous. Finally, $H/{\sim}$
is a connected separable compact line with more than one point; therefore it is order-isomorphic to $[0,1]$.
\end{proof}

\begin{rem}\label{thm:likemetrizable}
We note that separable compact lines share a lot of properties of metrizable spaces. Namely, a separable
compact line is first countable and perfectly normal (i.e., it is normal and every closed set is a $G_\delta$ set).
To see that a separable compact line is perfectly normal, note that every open set is the countable union of
its convex components, which are countable unions of closed intervals. The latter argument also shows
that every open set in a separable compact line is $\sigma$-compact and hence separable compact lines are
hereditarily Lindel\"of. Moreover, separability is hereditary for compact lines, as will be shown in the next lemma.
\end{rem}

\begin{lem}\label{thm:hereditseparable}
A separable compact line is hereditarily separable.
\end{lem}
\begin{proof}
Let $H$ be a separable compact line. Since $H$ is first-countable, it suffices to prove that
every closed subset $Z$ of $H$ is separable. Let $D$ be a countable dense subset of $H$ and for $t\in Z\setminus\{\max Z\}$ right isolated in the compact line $Z$, denote by $t'$ the immediate successor of $t$ in $Z$. Set:
\[S=\big\{t\in Z\setminus\{\max Z\}:\text{$t$ right isolated in $Z$ and $\left]t,\smash{t'}\right[\ne\emptyset$}\big\}.\]
The intervals $\left]t,\smash{t'}\right[$, $t\in S$, are nonempty and pairwise disjoint; thus, the separability of $H$
implies that $S$ is countable. We claim that the countable set:
\[(D\cap Z)\cup\{\min Z,\max Z\}\cup\bigcup_{t\in S}\{t,t'\}\subset Z\]
is dense in $Z$. Namely, given $t_1,t_2\in Z$ with $\left]t_1,t_2\right[\cap Z\ne\emptyset$,
if $\left]t_1,t_2\right[\subset Z$, then $\left]t_1,t_2\right[$ intersects $D\cap Z$. Otherwise,
$\left]t_1,t_2\right[$ intersects $\{t,t'\}$, for some $t\in S$.
\end{proof}

\begin{lem}\label{thm:Usequence}
Let $U$ be open relatively to $\left[0,1\right[$ and let $\varepsilon>0$ be fixed. There
exists a sequence $\big(\left[a_k,b_k\right[\big)_{k\ge1}$ of pairwise disjoint intervals
such that:
\begin{itemize}
\item[(a)] $U=\bigcup_{k=1}^\infty\left[a_k,b_k\right[$;
\item[(b)] $b_k-a_k<\varepsilon$, for all $k\ge1$;
\item[(c)] $b_k-a_k\longrightarrow0$.
\end{itemize}
\end{lem}
\begin{proof}
The thesis is trivial when $U$ is connected and the general case is obtained by writing $U$ as the union
of its connected components.
\end{proof}

\begin{lem}\label{thm:lemagatolegal}
If $B$ is a $G_\delta$ subset of $\left[0,1\right[$, then there exists a weak*-null sequence $(G_n)_{n\ge1}$
in $\NBV\big([0,1]\big)$ such that $B=\big\{u\in[0,1]:G_n(u)\notarrow0\big\}$.
\end{lem}
\begin{proof}
Write $B=\bigcap_{j=1}^\infty U_j$, with $(U_j)_{j\ge1}$ a decreasing sequence of sets open in $\left[0,1\right[$.
For each $j\ge1$, apply Lemma~\ref{thm:Usequence} with $U=U_j$ and $\varepsilon=\frac1j$, obtaining a sequence of intervals $\big(\left[a_{jk},b_{jk}\right[\big)_{k\ge1}$. Denote by $F_{jk}\in\NBV\big([0,1]\big)$ the characteristic function of $\left[a_{jk},b_{jk}\right[$, which corresponds through \eqref{eq:muFmu} to the measure $\delta_{a_{jk}}-\delta_{b_{jk}}\in M\big([0,1]\big)$. The desired sequence $(G_n)_{n\ge1}$ is defined by setting
$G_n=F_{j(n)k(n)}$, where $n\mapsto\big(j(n),k(n)\big)$ is an enumeration of all pairs of positive integers.
To see that $(G_n)_{n\ge1}$ is weak*-null, note that $b_{j(n)k(n)}-a_{j(n)k(n)}\longrightarrow0$.
\end{proof}

\goodbreak

\begin{lem}\label{thm:Gdeltaweakstar}
Let $H$ be a separable compact line and $A$ be a $G_\delta$ subset of $H$ consisting of internal points
of $H$. Then there exists a weak*-null sequence $(F_n)_{n\ge1}$ in $\NBV(H)$ such that $A$ is the union
of $\big\{t\in H:F_n(t)\notarrow0\big\}$ with a countable set.
\end{lem}
\begin{proof}
If $H$ is countable, take $F_n=0$, for all $n$. If $H$ is uncountable, pick $\psi:H\to[0,1]$ as in the statement of
Lemma~\ref{thm:lemapsi}. For $u\in[0,1]$, write $\psi^{-1}(u)=[a_u,b_u]$. The set:
\[E=\big\{u\in[0,1]:\vert\psi^{-1}(u)\vert>2\big\}\]
is countable, since $H$ is separable and the open intervals $\left]a_u,b_u\right[$, $u\in E$, are nonempty
and pairwise disjoint. The set $\psi^{-1}[E]$ is also countable. Since $A$ consists of internal points of $H$, we have:
\[A\subset\psi^{-1}\big[\left]0,1\right[\setminus Q(\psi)\big]\cup\psi^{-1}[E].\]
Then $A\setminus\psi^{-1}[E]=\psi^{-1}[B]$, for some $B\subset\left]0,1\right[\setminus Q(\psi)$.
Since $A\setminus\psi^{-1}[E]$ is a $G_\delta$ subset of $H$ and the map $\psi$ is closed and surjective, it follows that $B$ is a $G_\delta$ subset of $[0,1]$. Let $(G_n)_{n\ge1}$ be as in the statement of Lemma~\ref{thm:lemagatolegal}.
Set $F_n=G_n\circ\psi$, so that $F_n\in\NBV(H)$ and $\Vert F_n\Vert_{\BV}=\Vert G_n\Vert_{\BV}$.
To conclude the proof, it remains to show that $(F_n)_{n\ge1}$ is weak*-null. To this aim, we use Corollary~\ref{thm:corweakstar}. Denote by $\mu_n\in M(H)$ and $\nu_n\in M\big([0,1]\big)$ the measures
corresponding to $F_n$ and $G_n$, respectively, through \eqref{eq:muFmu}. We have that $\psi_*(\mu_n)=\nu_n$,
since $G_n(u)=F_n(b_u)$, for all $u\in[0,1]$. Thus $\big(\psi_*(\mu_n)\big)_{n\ge1}$ is weak*-null.
Finally, denoting by $R_u:H\to[a_u,b_u]$ the canonical retraction, we compute:
\[(R_u)_*(\mu_n)=G_n(u)\delta_{a_u}+\big(G_n(1)-G_n(u)\big)\delta_{b_u}\in M\big([a_u,b_u]\big).\]
For $u\in Q(\psi)$, we have that $u\not\in B$ and therefore $G_n(u)\longrightarrow0$. Hence
the sequence $\big((R_u)_*(\mu_n)\big)_{n\ge1}$ is norm-convergent to zero.
\end{proof}

By a $G_{\delta\sigma}$ (resp., $F_{\sigma\delta}$) subset of a topological space we mean a subset that
is a countable union (resp., countable intersection) of $G_\delta$ (resp., $F_\sigma$) subsets.
\begin{lem}\label{thm:Gdeltasigma}
Let $\mathfrak X$ be a topological space and $(F_n)_{n\ge1}$ be a sequence of maps $F_n:\mathfrak X\to\R$
having at most a countable number of discontinuity points. Then the set $\big\{p\in\mathfrak X:F_n(p)\notarrow0\big\}$
is the union of a $G_{\delta\sigma}$ subset of $\mathfrak X$ with a countable set.
\end{lem}
\begin{proof}
We have:
\[\big\{p\in\mathfrak X:F_n(p)\notarrow0\big\}=\bigcup_{m=1}^\infty\bigcap_{n=1}^\infty\bigcup_{k=n}^\infty
\big\{p\in\mathfrak X:\vert F_k(p)\vert>\tfrac1m\big\}.\]
To conclude the proof, note that the set $\big\{p\in\mathfrak X:\vert F_k(p)\vert>\frac1m\big\}$
is the union of its interior with some discontinuity points of $F_k$.
\end{proof}

\goodbreak

\begin{lem}\label{thm:zeroforadoH}
Let $(\mu_n)_{n\ge1}$ be a weak*-null sequence in $M(L)$ and let $H$ be a closed subset of $L$ containing
$\supp\mu_n$, for all $n\ge1$.
If $F_n\in\NBV(L)$ corresponds to $\mu_n$ through \eqref{eq:muFmu}, then $F_n(t)\longrightarrow0$, for all $t\in L\setminus H$.
\end{lem}
\begin{proof}
If $[0,t]\cap H=\emptyset$, then $F_n(t)=0$, for all $n\ge1$. Otherwise, let $s$ be the maximum of $[0,t]\cap H$.
Then $F_n(t)=\mu_n\big([0,s]\cap H\big)$, for all $n$. Note that $s$ is right isolated in $H$ and therefore
$[0,s]\cap H$ is clopen in $H$. Moreover, $(\mu_n\vert_H)_{n\ge1}$ is weak*-null in $M(H)$ and hence
$\mu_n\big([0,s]\cap H\big)\longrightarrow0$.
\end{proof}

\begin{lem}\label{thm:CA}
Let $H$ be a separable compact line and $C$ be a subset of $H$ consisting of internal points of $H$. Then there exists
a subset $A$ of $C$ such that $C\setminus A$ is countable and every point of $A$ is a two-sided limit point
of $A$ (relatively to $H$).
\end{lem}
\begin{proof}
Let $A$ denote the set of points of $C$ that are two-sided condensation points of $C$. The conclusion
will follow if we show that $C\setminus A$ is countable. We have $C\setminus A=S_+\cup S_-$, where
$S_+$ (resp., $S_-$) denotes the set of points of $C$ that are not right condensation (resp., left condensation)
points of $C$. Let us prove that $S_+$ is countable. For $t\in S_+$, let $t'>t$ be such that $C\cap\left]t,\smash{t'}\right[$
is countable. Note that $\left]t,\smash{t'}\right[\ne\emptyset$, since $t$ is an internal point of $H$.
Setting $W=\bigcup_{t\in S_+}\left]t,\smash{t'}\right[$, it is easily seen that the open intervals
$\left]t,\smash{t'}\right[$, with $t\in S_+\setminus W$, are pairwise disjoint. Hence, by the separability of $H$,
the set $S_+\setminus W$ is countable.
Finally, the fact that $H$ is hereditarily Lindel\"of (Remark~\ref{thm:likemetrizable}) implies that $C\cap W$
(and then also $S_+\cap W$) is countable.
\end{proof}

\begin{proof}[Proof of Theorem~\ref{thm:main}]
Assume that $\phi$ has the \EP\ and let $A$ be a separable $G_\delta$ subset of $L$ such that every point
of $A$ is a two-sided limit point of $A$. Let $H$ denote the closure of $A$, so that $H$ is a separable
compact line and $A$ is a $G_\delta$ subset of $H$ consisting of internal points of $H$. Pick a sequence $(F_n)_{n\ge1}$ in $\NBV(H)$ as in Lemma~\ref{thm:Gdeltaweakstar}. Let $\mu_n\in M(H)$ correspond to
$F_n$ through \eqref{eq:muFmu} and let $\bar\mu_n\in M(L)$ be the extension of $\mu_n$ that vanishes identically
outside of $H$. Then $(\bar\mu_n)_{n\ge1}$ is weak*-null in $M(L)$ and the function $\overline F_n\in\NBV(L)$
that corresponds to $\bar\mu_n$ through \eqref{eq:muFmu} is an extension of $F_n$. Since $\phi$ has the \EP, using Proposition~\ref{thm:extcriteria}
with the sequence $(\overline F_n)_{n\ge1}$, we obtain that
the set $\big\{t\in Q(\phi):\overline F_n(t)\notarrow0\big\}$ is countable. Hence $A\cap Q(\phi)$ is countable.
Conversely, let $T:C(L)\to c_0$ be a bounded operator associated with a weak*-null sequence $(\mu_n)_{n\ge1}$
in $M(L)$ and let $F_n\in\NBV(L)$ correspond to $\mu_n$ through \eqref{eq:muFmu}. By Proposition~\ref{thm:extcriteria},
in order to conclude the proof of the theorem, we need to show that the set
$\big\{t\in Q(\phi):F_n(t)\notarrow0\big\}$ is countable. A bounded variation function has only
a countable number of discontinuity points and thus, by Lemma~\ref{thm:Gdeltasigma}, we can write
$\big\{t\in L:F_n(t)\notarrow0\big\}$ as the union of $\bigcup_{m=1}^\infty C_m$ with a countable set,
where each $C_m$ is a $G_\delta$ subset of $L$. It remains to show that $C_m\cap Q(\phi)$ is countable,
for all $m\ge1$. It follows from \cite[Lemma~2.1]{KK} that $\supp\mu_n$ is separable, for all $n$, and hence
the closure $H$ of $\bigcup_{n=1}^\infty\supp\mu_n$ is a separable compact line. By Lemma~\ref{thm:zeroforadoH},
each $C_m$ is contained in $H$. Note that the sequence
$(\mu_n\vert_H)_{n\ge1}$ is weak*-null in $M(H)$ and that $F_n\vert_H$ corresponds to $\mu_n\vert_H$
through \eqref{eq:muFmu}; thus, by Lemma~\ref{thm:externalpoints}, each $C_m$ contains at most a countable
number of external points of $H$. From Lemma~\ref{thm:CA} we obtain a subset $A_m$ of $C_m$ such that
$C_m\setminus A_m$ is countable and every point of $A_m$ is a two-sided limit point of $A_m$.
Then $A_m$ is a $G_\delta$ subset of $L$ and it is separable, by Lemma~\ref{thm:hereditseparable}.
Hence our assumptions imply that $A_m\cap Q(\phi)$ is countable. This concludes the proof.
\end{proof}

\end{section}

\begin{section}{The separable case}
\label{sec:separable}

In the case when the compact line $L$ is separable, the condition presented in Theorem~\ref{thm:main} can be
simplified.
\begin{teo}\label{thm:mainseparavel}
Let $K$ and $L$ be compact lines, with $L$ separable. A continuous increasing surjection $\phi:K\to L$ has the \EP\ if and only if the following condition holds: for every $G_\delta$ subset $C$ of $L$, if every point of $C$ is an internal point of $L$, then $C\cap Q(\phi)=C\cap Q_0(\phi)$ is countable.
\end{teo}
\begin{proof}
Follows from Theorem~\ref{thm:main} and Lemmas~\ref{thm:hereditseparable} and \ref{thm:CA}.
\end{proof}

\begin{defin}
A continuous increasing surjection $\phi:K\to L$ is said to be of {\em countable type\/} if the set $Q_0(\phi)$ is countable.
\end{defin}
Clearly, if $\phi$ is of countable type, then it has the \EP. In Example~\ref{thm:exaBernstein} below we will see that the converse does not hold, even when $L$ is separable. First, we need some terminology.

Given a subset $X$ of $[0,1]$, we denote by $\DA(X)$ the set:
\[\DA(X)=\big([0,1]\times\{0\}\big)\cup\big(X\times\{1\}\big)\subset[0,1]\times\{0,1\}\]
endowed with the lexicographic order. Then $\DA(X)$ is a separable compact line whose
set of internal points is $\big(\left]0,1\right[\setminus X\big)\times\{0\}$.
The first projection $\pi_1:\DA(X)\to[0,1]$ is a continuous increasing surjection. More generally, to each inclusion
$Y\subset X\subset[0,1]$, there corresponds a continuous increasing surjection $\phi:\DA(X)\to\DA(Y)$
defined by $\phi(u,0)=(u,0)$, for $u\in[0,1]\setminus X$, $\phi(u,i)=(u,0)$, for $u\in X\setminus Y$, $i=0,1$,
and $\phi(u,i)=(u,i)$, for $u\in Y$, $i=0,1$. We have $Q(\phi)=(X\setminus Y)\times\{0\}$ and
$Q_0(\phi)=Q(\phi)\setminus\big\{(0,0),(1,0)\big\}$.

\begin{example}\label{thm:exaBernstein}
Let $X$ be a subset of $[0,1]$ such that $[0,1]\setminus X$ is uncountable, but $[0,1]\setminus X$ does not
contain any uncountable closed set (see the construction that appears in \cite[Example~8.24]{Kechris}). Then $[0,1]\setminus X$ does not
contain any uncountable Borel set (see \cite[Theorem~13.6]{Kechris}). Set $K=\DA\big([0,1]\big)$,
$L=\DA(X)$ and let $\phi:K\to L$ be the continuous increasing surjection corresponding to the inclusion
$X\subset[0,1]$. Then $\phi$ is not of countable type. Using Theorem~\ref{thm:mainseparavel}, we show that $\phi$ has the
\EP. Let $C$ be a $G_\delta$ subset of $\DA(X)$ consisting of internal points of $\DA(X)$, i.e.,
$C\subset\big(\left]0,1\right[\setminus X\big)\times\{0\}$. Then $C=\pi_1^{-1}[B]$, with $B\subset\left]0,1\right[\setminus X$ and, since $\pi_1$ is a surjective closed map, it follows that $B$ is a $G_\delta$ subset of $[0,1]$. In particular, $B$ is a Borel set and hence $C=B\times\{0\}$ is countable.
\end{example}

\goodbreak

The set $X$ used in Example~\ref{thm:exaBernstein} to construct the compact line $L$ must be quite strange: it cannot be a Borel set and, under the assumption of existence of certain large cardinals, it cannot even belong to the projective hierarchy (see \cite[Theorem~38.17]{Kechris} and \cite{PD}). It is natural to ask whether, for
a ``sufficiently regular'' separable compact line $L$, every continuous increasing surjection $\phi:K\to L$
with the \EP\ is of countable type. We consider the following notion of regularity for separable compact lines.
\begin{defin}\label{thm:defBorelregular}
A separable compact line $L$ is said to be {\em Borel regular\/} if its set of internal points is a Borel subset of $L$.
\end{defin}

Definition~\ref{thm:defBorelregular} is motivated by the following result.
\begin{prop}\label{thm:Borelregular}
If $X$ is a subset of $[0,1]$, then $\DA(X)$ is Borel regular if and only if $X$ is a Borel set.
\end{prop}

The proof of Proposition~\ref{thm:Borelregular} uses the lemma below. Recall that the {\em Baire $\sigma$-algebra\/}
of a topological space is the $\sigma$-algebra spanned by the zero sets of continuous real-valued functions
on that space.
\begin{lem}\label{thm:increasing}
A (not necessarily continuous) increasing map $\lambda:K\to L$ is measurable, if $K$ is endowed with its Borel $\sigma$-algebra and $L$ is endowed
with its Baire $\sigma$-algebra. In particular, if $L$ is separable, then $\lambda$ is measurable when both $K$ and $L$ are endowed with their Borel
$\sigma$-algebras.
\end{lem}
\begin{proof}
Since the Baire $\sigma$-algebra of $L$ is the smallest $\sigma$-algebra for which every element of $C(L)$ is measurable, it is sufficient
to show that $f\circ\lambda:K\to\R$ is Borel measurable, for every $f\in C(L)$. Moreover, by \cite[Proposition~3.2]{Kubis}, the set of continuous increasing functions is linearly dense in $C(L)$. Therefore, we assume without loss of generality that $f$ is increasing. Then, for every $c\in\R$,
the set $\big\{s\in K:(f\circ\lambda)(s)\le c\big\}$ is an interval of $K$ and hence a Borel set. Finally, note that if $L$ is separable
then $L$ is perfectly normal (Remark~\ref{thm:likemetrizable}) and hence its Borel and Baire $\sigma$-algebras coincide.
\end{proof}

\begin{proof}[Proof of Proposition~\ref{thm:Borelregular}]
If $X$ is a Borel
subset of $[0,1]$, then the set of internal points of $\DA(X)$ is a Borel set, being the inverse image under the continuous map $\pi_1:\DA(X)\to[0,1]$ of $\left]0,1\right[\setminus X$.
Conversely, assume that $\DA(X)$ is Borel regular. The set $\left]0,1\right[\setminus X$ is the inverse image
under the increasing map $\lambda:[0,1]\ni u\mapsto(u,0)\in\DA(X)$ of the set of internal points of $\DA(X)$. Hence, by Lemma~\ref{thm:increasing}, $X$ is a Borel subset of $[0,1]$.
\end{proof}

We will finish the section by proving that, for a separable Borel regular compact line $L$, the statement
\begin{equation}\label{eq:statement}
\text{$\phi$ has the \EP}\Longleftrightarrow\text{$\phi$ is of countable type}
\end{equation}
is independent of the axioms of ZFC. More specifically, in Proposition~\ref{thm:underCH} we will show that,
under the continuum hypothesis (CH), the equivalence \eqref{eq:statement} is false. Moreover, Proposition~\ref{thm:underMA} shows that, assuming Martin's Axiom and the negation of CH (more precisely,
assuming $\text{MA}(\omega_1)$), the equivalence \eqref{eq:statement} holds.

\begin{prop}\label{thm:underCH}
Assume CH. There exist separable compact lines $K$ and $L$, with $L$ Borel regular, and a continuous increasing surjection $\phi:K\to L$
that has the \EP, but is not of countable type.
\end{prop}
\begin{proof}
Let $X$ be a Borel subset of $[0,1]$ that is not an $F_{\sigma\delta}$ set (see \cite[Theorem~22.4]{Kechris}).
Under CH, the collection of all $G_\delta$ subsets of $[0,1]$ contained in $\left]0,1\right[\setminus X$
can be written as $\big\{B_\alpha:\alpha<\omega_1\big\}$. Define, by recursion, a family $(u_\alpha)_{\alpha<\omega_1}$ of points of $\left]0,1\right[\setminus X$ such that, for all $\alpha<\omega_1$, we have that $u_\alpha\not\in\bigcup_{\beta<\alpha}\big(B_\beta\cup\{u_\beta\}\big)$. This is possible, because
$\left]0,1\right[\setminus X$ is not a $G_{\delta\sigma}$ set. Define $S=\big\{u_\alpha:\alpha<\omega_1\big\}$,
$K=\DA(X\cup S)$, $L=\DA(X)$, and let $\phi:K\to L$ be the continuous increasing surjection corresponding
to the inclusion $X\subset X\cup S$.
By Proposition~\ref{thm:Borelregular}, the separable compact line $L$ is Borel regular. Moreover, $Q_0(\phi)=S\times\{0\}$, so that $\phi$ is not of countable type. Using Theorem~\ref{thm:mainseparavel}, we show that $\phi$ has the \EP. Let $C$ be a $G_\delta$ subset of $\DA(X)$ consisting of internal points of $\DA(X)$.
As argued in Example~\ref{thm:exaBernstein}, we have $C=B\times\{0\}$, with $B$ a $G_\delta$
subset of $[0,1]$ contained in $\left]0,1\right[\setminus X$. Then $B=B_\alpha$, for some $\alpha<\omega_1$.
To conclude the proof, note that $C\cap Q_0(\phi)=(B\cap S)\times\{0\}$ and that $B\cap S\subset\big\{u_\beta:\beta\le\alpha\big\}$.
\end{proof}

The next lemma is used in the proof of Proposition~\ref{thm:underMA}.
\begin{lem}\label{thm:MAkappa}
Assume $\text{MA}(\kappa)$. Every subset $M$ of $\omega^\omega$ with $\vert M\vert\le\kappa$ is contained in a $\sigma$-compact subset of $\omega^\omega$.
\end{lem}
\begin{proof}
Define a partial order $\le$ on $\omega^\omega$ by stating that $(x_n)_{n\in\omega}\le(y_n)_{n\in\omega}$
if and only if $x_n\le y_n$, for all $n\in\omega$. Clearly, the relatively compact subsets of $\omega^\omega$
coincide with the $\le$-bounded subsets of $\omega^\omega$. Now, define a preorder $\le^*$ on $\omega^\omega$
by stating that $(x_n)_{n\in\omega}\le^*(y_n)_{n\in\omega}$ if and only if $x_n\le y_n$, for all but finitely
many $n\in\omega$. It is easily checked that $M$ is contained in a $\sigma$-compact subset of $\omega^\omega$
if and only if $M$ is $\le^*$-bounded. The conclusion follows from the fact that, under $\text{MA}(\kappa)$,
every subset $M$ of $\omega^\omega$ with $\vert M\vert\le\kappa$ is $\le^*$-bounded (see \cite[(8) pg.\ 87]{Kunen}).
\end{proof}

\begin{prop}\label{thm:underMA}
Assume $\text{MA}(\omega_1)$. Let $K$ and $L$ be compact lines and $\phi:K\to L$ be a continuous increasing surjection. Assume that $L$ is separable and
Borel regular. If $\phi$ has the \EP, then $\phi$ is of countable type.
\end{prop}
\begin{proof}
The conclusion is trivial if $L$ is countable. If $L$ is uncountable, let $\psi:L\to[0,1]$ be as in the statement
of Lemma~\ref{thm:lemapsi}. Arguing as in the beginning of the proof of Lemma~\ref{thm:Gdeltaweakstar},
one obtains that the set of internal points of $L$ is the union of $\psi^{-1}\big[\left]0,1\right[\setminus Q(\psi)\big]$ with a countable set. Then $Q_0(\phi)$ is the union of $\psi^{-1}[S]$ with a countable set,
where $S$ is a subset of $\left]0,1\right[\setminus Q(\psi)$.

Now assume by contradiction that $Q_0(\phi)$ is uncountable, so that $S$ is uncountable. We will prove that $\phi$ does not have the \EP\ by exhibiting a $G_\delta$ subset $C$ of $L$, consisting of internal points of $L$, such that $C\cap Q_0(\phi)$ is uncountable. First, we claim that $\left]0,1\right[\setminus Q(\psi)$ is a Borel set. Namely,
the Borel regularity of $L$ implies that
$\psi^{-1}\big[\left]0,1\right[\setminus Q(\psi)\big]$ is a Borel subset of $L$. The claim then follows
from Lemma~\ref{thm:increasing}, noting that $\left]0,1\right[\setminus Q(\psi)$ is the inverse image of $\psi^{-1}\big[\left]0,1\right[\setminus Q(\psi)\big]$ under an increasing right inverse $\lambda:[0,1]\to L$
of $\psi$.

The Borel set $\left]0,1\right[\setminus Q(\psi)$ is the image
of a continuous map $\theta:\omega^\omega\to[0,1]$ (see \cite[Theorem~13.7]{Kechris}). Since $S$ is an uncountable
subset of the image of $\theta$,
there exists a subset $M$ of $\omega^\omega$ such that $\vert M\vert=\omega_1$, $\theta\vert_M$ is injective
and $\theta[M]\subset S$. By Lemma~\ref{thm:MAkappa}, $M$ is contained in a $\sigma$-compact subset of $\omega^\omega$
and therefore there exists a compact subset $\mathcal K$ of $\omega^\omega$ with $\vert M\cap\mathcal K\vert=\omega_1$.
Then $\theta[\mathcal K]$ is a compact subset of $\left]0,1\right[\setminus Q(\psi)$ and $S\cap\theta[\mathcal K]$ is uncountable. The conclusion is now obtained by taking $C=\psi^{-1}\big[\theta[\mathcal K]\big]$.
\end{proof}

\end{section}

\end{document}